\documentclass[preprint,10pt]{article}
\usepackage[a4paper,margin=2.0cm]{geometry}

\usepackage{graphicx}
\usepackage{amsmath}
\usepackage{amssymb}
\usepackage{amsfonts}
\usepackage{amsthm}
\usepackage{subcaption}
\usepackage{ifthen}
\usepackage{hyperref}
\usepackage{authblk}
\usepackage{enumerate}

\newcommand{\fa}{\hbox{ for all }}
\newcommand{\be}{\beta,\varepsilon}
\newcommand{\R}{\mathbb{R}}
\newcommand{\N}{\mathbb{N}}
\newcommand{\calh}{\mathcal H}
\newcommand{\inner}[2]{\ifthenelse{\equal{#2}{}}{\left\langle\cdot,\cdot\right\rangle_{#1}}{\left\langle#2\right\rangle_{#1}}}
\newcommand{\norm}[2]{\ifthenelse{\equal{#2}{}}{\left\|\cdot\right\|_{#1}}{\left\|#2\right\|_{#1}}}
\newcommand{\seminorm}[2]{\ifthenelse{\equal{#2}{}}{\left|\cdot\right|_{#1}}{\left|#2\right|_{#1}}}
\DeclareMathOperator{\Null}{Null}
\DeclareMathOperator{\Sp}{span}
\DeclareMathOperator*{\argmax}{arg\max}

\newtheorem{theorem}{Theorem}
\newtheorem{prop}[theorem]{Proposition}
\newtheorem{lemma}[theorem]{Lemma}

\graphicspath{{./}}

\title{Equally spaced points are optimal for Brownian Bridge kernel interpolation}

\author[1]{Gabriele Santin
\thanks{gsantin@fbk.eu, 
\href{https://gabrielesantin.github.io/}{gabrielesantin.github.io/}}}
\affil[1]{Digital Society Center, Bruno Kessler Foundation (Trento, Italy)}
\date{\today}

\begin{document}

\maketitle

\begin{abstract}
In this paper we show how ideas from spline theory can be used to construct a local basis for the space of translates of a general iterated Brownian 
Bridge kernel $k_{\beta,\varepsilon}$ for $\beta\in\N$, $\varepsilon\geq 0$.
In the simple case $\beta=1$, we derive an explicit formula for the corresponding Lagrange 
basis, which allows us to solve interpolation problems without inverting any linear system. 

We use this basis to prove that interpolation with $k_{1,\varepsilon}$ is uniformly stable, i.e., the Lebesgue constant is bounded 
independently of the number an location of the interpolation points, and that equally spaced points are the unique minimizers of the associated 
power function, and are thus error optimal. 
In this derivation, we investigate the role of the shape parameter $\varepsilon>0$, and discuss its effect on these error and stability bounds.

Some of the ideas discussed in this paper could be extended to more general Green kernels.

\end{abstract}

\section{Introduction}
Positive definite kernels are used in a variety of settings and applications to solve various approximation problems \cite{Fasshauer2015,Wendland2005}. 
They are the basis of several algorithms, and their theory is well understood.

In the theory of Partial Differential Equations (PDE), one may end up studying positive definite kernels as Green kernels of certain differential 
operators. 
In recent years, a novel point of view has emerged (see \cite{Fasshauer2011a,Fasshauer2012a,Fasshauer2013}, and Chapter 6 in \cite{Fasshauer2015}), that 
allows one to follow the path in the opposite direction, i.e., designing PDEs starting from a given kernel.
This approach is particularly attractive to study the original interpolation problem because one may hope, to some extent, to derive properties of the kernel 
interpolant by looking solely at the associated differential operator (including its boundary conditions) and the domain of definition. 

We are interested in trying to connect this point of view to the problem of optimal sampling: Given the freedom to choose a number of
interpolation points, where should they be placed to minimize certain indicators? This question is of central importance in approximation, and we aim 
at understanding if the differential operator and the boundary conditions of a Green kernel may be used to infer the optimality of certain point distributions.

As an initial investigation in this direction, in this paper we consider the family of Brownian Bridge kernels $k_{1, \varepsilon}$, $\varepsilon\geq 0$ 
(Section \ref{sec:bb_basic}), which allow us to carry out explicit computations and 
derive formulas for some key quantities, including the Lagrange basis (Section \ref{sec:results}), the Lebesgue constant (Section \ref{sec:lebesgue}), 
and the power function (Section \ref{sec:power}). These are all very elementary computations, which permit 
nevertheless to spell out explicitly some interesting aspects. In particular, we prove that equally spaced points are error-optimal (Theorem \ref{thm:power}), 
and we provide an explicit rate of convergence of the interpolant in the flat limit $\varepsilon\to 0$ (Theorem \ref{thm:flat}).

This initial analysis exploits the connection with existing and well-established results in spline theory, which may be the ground for an 
extension of these result to more general kernels, at least in one dimension.

\section{The Brownian Bridge kernels}\label{sec:bb_basic}
We recall some facts on the Brownian Bridge kernels, and refer to \cite{CaFaMcC2014}, and Chapter 7 in \cite{Fasshauer2015}, for further details.
We denote as $D^j u(x) := u^{(j)}(x):=\partial_x^j u(x)$, $j\in \N_0$, the $j$-order weak derivative of $u$ w.r.t. $x$, and for $\beta\in\N$ we define the 
Sobolev space  
$H^\beta_0(0,1)$ with $u\in H^\beta_0(0,1)$ if and only if $D^j u\in L_2(0,1)$, $0\leq j\leq \beta$, and $D^j u_{|\{0,1\}}=0$, $0\leq j\leq \beta-1$.
Given $\beta\in\N$, $\varepsilon\geq 0$, $f\in L_2(0,1)$, the Brownian bridge kernel $k_{\varepsilon,\beta}:(0,1)^2\to\R$ is the Green kernel of the PDE
\begin{equation}\label{eq:PDE}
\begin{cases}
L_{\beta, \varepsilon} u(x) :=\left(-D^2 + \varepsilon D^0\right)^\beta u(x) &= f(x),\;\; x\in (0,1)\\
u^{(j)}(0)=u^{(j)}(1) &= 0, \;\; 0\leq j\leq \beta-1,
\end{cases}
\end{equation}
and in it is strictly positive definite on $(0,1)$.
Since $k_{\be}$ is the Green kernel of the PDE \eqref{eq:PDE}, setting $N_{\be}:=\Null(L_{\be})$ we have 
\begin{equation}\label{eq:green_kernel_properties}
k_{\beta, \varepsilon}(\cdot, y)\in H^\beta_0(0,1), \quad 
k_{\be}(\cdot, y)_{|(0,y)}\in N_{\be}, \quad
k_{\be}(\cdot, y)_{|(y,1)}\in N_{\be}
\fa y\in (0,1).
\end{equation}
Moreover, if $T_{\be}:=\left(\varepsilon D^0 + D^1\right)^\beta$ then for all $u\in H^1_0(0,1)$ and $f:=L_{\beta, \varepsilon} u\in L_2(0,1)$ we have that
\begin{equation*}
u(x) 
= \int_0^1 k_{\be}(x, y) f(y) dy 
= \int_0^1 T_{\be} k_{\be}(x, y) T_{\be} u(y) dy=:\inner{\be}{k_{\be}(\cdot, x),u}.
\end{equation*}
This means that $k_{\beta, \varepsilon}$ is a reproducing kernel on $H^\beta_0(0,1)$ with respect to the inner product $\inner{\be}{u,v}$, i.e., 
$(H^\beta_0(0,1), \inner{\be}{})$ is the unique Reproducing Kernel Hilbert Space (RKHS) or native space of $k_{\beta, \varepsilon}$ on $(0,1)$. 

It is known that $x\mapsto k_{\beta, 0}(x, y)$, $y\in (0,1)$, is a piecewise polynomial of degree $2\beta-1$ on $(0,y]$ and $[y, 1)$. 
Moreover, for any $\beta$ and $\varepsilon$ the kernels $k_{\beta, \varepsilon}$ can be represented in terms of their Mercer series, which is explicitly known 
and used in the RBF-QR 
or Hilbert-Schmidt SVD algorithms (see e.g. \cite{CaFaMcC2014}). 
For our purposes it suffices instead to recall that for $\beta=1,2$ there are closed-form expressions, 
and in particular
\begin{align}
k_{1, \varepsilon}(x,y) &= \begin{cases}
                              \min(x,y) - xy,& \varepsilon=0\\
                              \frac{\sinh(\varepsilon \min(x,y))\sinh(\varepsilon(1-\max(x,y)))}{\varepsilon\sinh(\varepsilon)},&\varepsilon>0
                              \end{cases}\label{eq:bb1_explicit},\;\;x,y\in(0,1).
\end{align}

\section{Interpolation and local bases}\label{sec:b_splines}
Given $N\in\N$ interpolation points $X:=\{x_1<x_2<\dots<x_N\}\subset (0,1)$, we set $x_0:=0$, $x_{N+1}:=1$. We further set $h_i:=x_{i+1}-x_i$, $0\leq i\leq N$.
It is known (see e.g. \cite{Fasshauer2015,Wendland2005}) that there exists a unique $k_{\be}$-interpolant $I_{\be,X} f$ of $f\in C(0,1)$ 
at $X$, 
with 
\begin{equation}\label{eq:cardinal_interp}
I_{\be,X} f(x) = \sum_{i=1}^N f(x_i) \ell_{\be, i}(x),
\end{equation}
and where $\{\ell_{\be, i}\}_{i=1}^N$ is the cardinal basis of 
$V_{\be}(X):=\Sp\{k_{\be}(\cdot, x): x\in X\}\subset \calh_{\be}$.
It follows that
$I_{\be,X} f\in V_{\be}(X)$, and using \eqref{eq:green_kernel_properties} we have also that
\begin{equation*}
V_{\be}(X)\subset S_{{\be}}(X):=\{f\in H^\beta_0(0,1): f_{|(x_i, x_{i+1})}\in N_{\be}, 0\leq i\leq N\}.
\end{equation*}
The elements of $S_{{\be}}(X)$ are called $L_{\be}$-splines (see \cite{Jerome1976,Jia1981,Seatzu1975}), and since this space is uniquely determined by 
the nodal values at $X$, it actually holds that $V_{\be}(X)= S_{{\be}}(X)$ by a dimension argument.

$L$-splines are well studied for a general differential operator $L$, and in particular there is a local basis of minimal support, named an 
$LB$-spline basis, which depends on $L$ and $X$ only (see \cite{Jerome1976,Jia1981,Seatzu1975} for different equivalent constructions). 
In our case $S_{\beta,0}(X)$ is the space of classical piecewise polynomial splines of degree $2\beta-1$, and the $LB$-splines are the usual 
$B$-splines. For $\varepsilon>0$ instead, following \cite{Seatzu1975} we consider the homogeneous linear problem 
$L_{\beta, \varepsilon} u= 0$, $\varepsilon>0$, which has a characteristic equation $(-z^{2} + \varepsilon^2)^\beta = 0$ with zeros $z_1 = \varepsilon$ and 
$z_2=-\varepsilon$, each with multiplicity $\beta$. It follows that a basis of $\Null(L_{\be})$
is given by the $2\beta$ linearly independent functions
\begin{equation*}
u_{j}(x):= x^{j-1} e^{\varepsilon x},\;\;u_{\beta+j}(x):= x^{j-1} e^{-\varepsilon x},\;\; 1\leq j\leq \beta.
\end{equation*}
With these we define $g(t, x):=\sum_{j=1}^{2 \beta} b_{j}(t) u_{j}(x)$, $t,x\in(0,1)$, with $b_{j}:(0,1)\to\R$, $1\leq j\leq 2\beta$ such that 
\begin{equation}\label{eq:g}
\left(\partial^i_x g(t,x)\right)_{|t=x} = 0, \;\;0\leq i\leq 2(\beta -1), \quad\left(\partial^{2\beta - 1}_x g(t,x)\right)_{|t=x} = (-1)^\beta,
\end{equation}
i.e., the vector $b(t):=(b_i(t))_{i=1}^{2\beta}$ solves the linear system $W(t) \cdot b(t) = e$ with $e:= (0,\dots, 0, (-1)^\beta)^T$,  where the matrix
$W(t):=\left(\partial_x^i u_j(t)\right)_{i,j=1}^{2\beta}\in\R^{2\beta\times 2\beta}$ is invertible since it is the Wronskian matrix of a set of 
linearly independent functions. From $g$ one can define the one-sided splines 
$g_+(t, x):= \chi_{(x, 1)}(t) g(t,x)$ and $g_-(t, x):= \chi_{(0, x)}(t) g(x,t)$, with $\chi_{(0, x)}$ the indicator function of $(0,x)$. 
The conditions \eqref{eq:g} and the definition of $g$ ensure that $g_+(x_i, \cdot), g_-(x_i, \cdot)\in S_{\be}(X)$ for all $x_i\in X$.

These one sided splines may be combined to obtained the $LB$-spline basis for $\varepsilon>0$ (see Theorem 3.1 in \cite{Seatzu1975}), so that each element of 
the basis has support in $2\beta$ consecutive intervals $[x_{i-1}, x_{i+2\beta-1}]$. This construction is not straightforward in general, but it is easier for 
$\beta=1$ and we spell out the details in the next section. 
Moreover, observe that for $\beta=1$ the support condition implies that the $i$-the element of the local basis is zero on $X\setminus\{x_i\}$, and it can thus 
be scaled to be $\ell_{\be, i}$. 

We will make use of the following basic identities, which are recalled for the reader's convenience.
\begin{lemma}\label{lemma:identities}
If $a,b,c,d\in\R$, then 
\begin{enumerate}[i)]
\item\label{item:first_identity} $\sinh(a) + \sinh(b) = 2\sinh((a+b)/2)\cosh((a-b)/2)$,
\item\label{item:fourth_identity} $\sinh(a-b)\sinh(c-d) - \sinh(a-c) \sinh(b-d) = \sinh(a-d)\sinh(c-b)$,
\item\label{item:nth_identity}  $\cosh(a)\sinh(b)-\sinh(a)\cosh(b) = \sinh(a-b)$.
\item\label{item:fifth_identity} $\sinh(a/2)^2/\sinh(a)=\tanh(a/2)/2$,
\end{enumerate}
\end{lemma}

\subsection{The Lagrange basis for $\beta=1$}
For $\beta=1$ and $\varepsilon>0$ we have $u_{1}(x) = e^{\varepsilon x}$, $u_{2}(x) = e^{-\varepsilon x}$, so that
\begin{equation*}
W(t) b(t)= \begin{bmatrix}
    \phantom{\varepsilon} e^{\varepsilon t} &\phantom{-\varepsilon}e^{-\varepsilon t}\\ \varepsilon e^{\varepsilon t} & -\varepsilon e^{-\varepsilon t}
    \end{bmatrix}
    \begin{bmatrix}b_1(t)\\b_2(t)\end{bmatrix} 
    = \begin{bmatrix}\phantom{-}0\\-1\end{bmatrix}
    \Leftrightarrow 
    b(t) =\frac1{2\varepsilon} \begin{bmatrix}-e^{-\varepsilon t}\\ e^{\varepsilon t}\end{bmatrix},
\end{equation*}
and thus
\begin{equation}\label{eq:g_beta_1}
g(t, x) 
= \frac1{2\varepsilon}\left(-e^{-\varepsilon t} e^{\varepsilon x} + e^{\varepsilon t}e^{-\varepsilon x}\right)
= \frac1{2\varepsilon}\left(e^{\varepsilon (t-x)}-e^{-\varepsilon (t-x)}\right)
= \frac{\sinh\left(\varepsilon (t-x)\right)}{\varepsilon}.
\end{equation}
The $i$-th basis element is $\ell_{\be,i}(x) = c_1 g_-(x_{i-1}, x) + c_2 g_-(x_i, x)$, and to enforce the cardinal conditions we require
\begin{equation*}
\begin{bmatrix}
g_-(x_{i-1}, x_i) & 0\\
g_-(x_{i-1}, x_{i+1}) & g_-(x_{i}, x_{i+1})
\end{bmatrix}
\begin{bmatrix}
c_1\\c_2
\end{bmatrix}
=
\begin{bmatrix}
1\\0
\end{bmatrix},
\end{equation*}
which can be solved by forward substitution to get
\begin{equation*}
c_1=\frac{\varepsilon}{\sinh\left(\varepsilon (x_i-x_{i-1})\right)}
=\frac{\varepsilon}{\sinh\left(\varepsilon h_{i-1}\right)}, \;\;
c_2=-\frac{\varepsilon \sinh\left(\varepsilon (x_{i+1}-x_{i-1})\right)}{\sinh\left(\varepsilon 
(x_i-x_{i-1})\right)\sinh\left(\varepsilon (x_{i+1}-x_{i})\right)}
=-\frac{\varepsilon \sinh\left(\varepsilon (h_i+h_{i-1})\right)}{\sinh\left(\varepsilon 
h_{i-1}\right)\sinh\left(\varepsilon h_i\right)}.
\end{equation*}
Combing this with \eqref{eq:g_beta_1}, we obtain 
\begin{equation}\label{eq:leb_1}
\ell_{1,\varepsilon, i}(x)
=c_1 g_-(x_{i-1}, x) + c_2 g_-(x_i, x)
=\begin{cases}
\frac{\sinh(\varepsilon (x - x_{i-1}))}{\sinh(\varepsilon h_{i-1})},& x_{i-1}\leq x<x_i\\
\frac{\sinh(\varepsilon (x_{i+i}-x))}{\sinh(\varepsilon h_{i})},& x_{i}\leq x<x_{i+1},
\end{cases},\;\;1\leq i\leq N.
\end{equation}
where in the second case we used \eqref{item:fourth_identity} of Lemma \ref{lemma:identities}.

\section{Stability and accuracy}\label{sec:results}
We use this basis to study the interpolant for $\varepsilon>0$, and comment on the limiting case 
$\varepsilon\to0$ in Section \ref{sec:flat}.
The Lebesgue function and the squared power function are defined as
\begin{equation}\label{eq:power_function}
\Lambda_{\beta, \varepsilon,X}(x):= \sum_{i=1}^N \left|\ell_{\beta, \varepsilon, i}(x)\right|,
\quad\quad
P_{\beta, \varepsilon, X}(x)^2:= k_{\beta, \varepsilon}(x, x) - \sum_{i=1}^N \ell_{\be,i}(x) k_{\beta, \varepsilon}(x, x_i),
\quad\quad x\in (0,1).
\end{equation}
The Lebesgue function provides the stability bound
\begin{equation*}
\left|I_{\beta, \varepsilon,X}f(x)\right|\leq \Lambda_{\beta, \varepsilon,X}(x) \norm{\infty}{f_{|X}}, \;\;x\in (0,1),
\end{equation*}
and if additionally $f\in \calh_{\be}=H_0^\beta(0,1)$, then the power function gives the error bound
\begin{equation}\label{eq:pf_bound}
\left|f(x) - I_{\beta, \varepsilon,X}f(x)\right|\leq P_{\beta, \varepsilon,X}(x) \norm{\calh_{\be}}{f}, \;\;x\in (0,1).
\end{equation}
We derive explicit bounds for $\Lambda_{1, \varepsilon,X}$ and $P_{1, \varepsilon, X}$ using the expression of the cardinal 
functions derived above. 

\subsection{Lebesgue function}\label{sec:lebesgue}
We have the following explicit formula for the Lebesgue function.
\begin{prop}\label{prop:lebesgue}
Let $X\subset (0,1)$ and $\varepsilon>0$. Then 
\begin{equation}\label{eq:lebesgue_explicit}
\Lambda_{1, \varepsilon,X}(x)
= 
\begin{cases}
\frac{\sinh(\varepsilon x)}{\sinh(\varepsilon h_0)},&0< x\leq x_1\\
2\ \frac{\sinh(\varepsilon h_i/2)}{\sinh(\varepsilon h_i)} \cosh\left(\varepsilon 
\left(x-\frac{x_{i}+x_{i+1}}{2}\right)\right),&x_i\leq x \leq x_{i+1}, 1\leq i\leq N-1\\
\frac{\sinh(\varepsilon (1 - x))}{\sinh(\varepsilon h_N)},& x_N\leq x< 1.
\end{cases}
\end{equation}
\end{prop}
\begin{proof}
Using the formula \eqref{eq:leb_1}, one gets for $x\in(0,x_1]$ or $x\in [x_N,1)$ that
\begin{equation*}
\Lambda_{1, \varepsilon,X}(x)
= \left|\ell_{1, \varepsilon, 1}(x)\right|=\frac{\sinh(\varepsilon x)}{\sinh(\varepsilon h_0)}, x\in(0,x_1],
\quad
\Lambda_{1, \varepsilon,X}(x)
=\left|\ell_{1, \varepsilon, N}(x)\right|=\frac{\sinh(\varepsilon (1-x))}{\sinh(\varepsilon h_{N})}, x\in[x_N, 1).
\end{equation*} 
If instead $x_i< x< x_{i+1}$ with $i\in{1\leq i\leq N-1}$, we use the fact that $\sinh$ is an odd function, and the formula 
\eqref{item:first_identity} in Lemma \ref{lemma:identities} to derive that
\begin{equation*}
\Lambda_{1, \varepsilon,X}(x)
=2\ \frac{\sinh(\varepsilon h_i/2)}{\sinh(\varepsilon h_i)} \cosh\left(\varepsilon \left(x-\frac{x_{i}+x_{i+1}}{2}\right)\right),
\end{equation*}
which is the desired expression.
\end{proof}
With this result, we can easily conclude that $\norm{L_\infty(0,1)}{\Lambda_{1, \varepsilon,X}}= 1$ for any of $X$ and $\varepsilon>0$, and 
thus interpolation with $k_{1,\varepsilon}$ is stable for all $\varepsilon>0$ and for any number of arbitrarily located interpolation points. 
\begin{theorem}\label{thm:lebesgue}
For any $X\subset (0,1)$ and $\varepsilon>0$ we have $\Lambda_{1, \varepsilon,X}(x)< 1$ for $x\in (0,1)\setminus X$, $\Lambda_{1, \varepsilon,X}(x) = 1$ 
for $x\in X$, and $\lim_{x\to0}\Lambda_{1, \varepsilon,X}(x)=\lim_{x\to 1}\Lambda_{1, \varepsilon,X}(1)=0$.
\end{theorem}
\begin{proof}
The fact that $\Lambda_{1, \varepsilon,X}=1$ on the interpolation points is well known and follows from the definition.
If $x\in(0,x_1]\cup [x_N,1)$, from equation \eqref{eq:lebesgue_explicit} we have $\Lambda_{1, 
\varepsilon,X}(x)\leq 1$ since $\sinh$ is increasing. The same formula can be used to verify the value of the limits as $x\to 0, 1$.
For $x\in[x_i,x_{i+1}]$ with $1\leq i\leq N-1$ we use the fact that $\cosh$ is an even function with a unique global minimum at $0$, so that 
\eqref{eq:lebesgue_explicit} shows that in $[x_i, x_{i+1}]$ the function $\Lambda_{1, 
\varepsilon,X}(x)$ has maxima in $x_i$ and $x_{i+1}$, where its value is $1$.
\end{proof}

\subsection{Power function}\label{sec:power}
We proceed similarly to derive and explicit formula for the power function.
\begin{prop}\label{prop:power}
Let $X\subset (0,1)$ and $\varepsilon>0$. Let $x_0:=0$, $x_{N+1}:=1$, and define
\begin{equation*}
x^+:=\argmax\{x_i: x_i\leq x, 0\leq i\leq N\},
\;\;
x^+:=\argmax\{x_i: x\leq x_i, 1\leq i\leq N+1\}.
\end{equation*}
Then 
\begin{equation}\label{eq:power_explicit}
P_{1, \varepsilon,X}(x)
=\sqrt{\frac{\sinh(\varepsilon(x-x^-)\sinh(\varepsilon(x^+-x))}{\varepsilon\sinh(\varepsilon (x^+ - x^-))}}, \;\;x\in(0,1).
\end{equation}
\end{prop}
\begin{proof}
To simplify the proof, we observe that \eqref{eq:leb_1} can be used to define $\ell_{1,\varepsilon,0}(x)$ for $x\in(x_0,x_1]$, and 
$\ell_{1,\varepsilon,N+1}(x)$ for $x\in(x_N,x_{N+1}]$. With these definitions, it follows from \eqref{eq:power_function} that if 
$x\in(x_i,x_{i+1})$ with $0\leq i\leq N$ we have
\begin{equation}\label{eq:power_tmp3}
P_{1,\varepsilon,X}(x)^2
=k_{\beta, \varepsilon}(x, x) - \ell_i(x) k_{\beta, \varepsilon}(x, x_i)-\ell_{i+1}(x) k_{\beta, \varepsilon}(x, x_{i+1}),
\end{equation}
where the formula is valid also for $x\in(x_0,x_1]\cup [x_N,x_{N+1})$, since $k_{\beta, \varepsilon}(x, 
x_{0})=k_{\beta, \varepsilon}(x,1)=0$ and thus in this case only two terms in the right hand side are nonzero.

We start by using \eqref{eq:bb1_explicit} and 
\eqref{eq:leb_1} to compute
\begin{align}\label{eq:power_tmp1}
k_{\beta, \varepsilon}(x, x)&-\ell_i(x) k_{\beta, \varepsilon}(x, x_i)
=\frac{\sinh(\varepsilon x)\sinh(\varepsilon(1-x))}{\varepsilon\sinh(\varepsilon)} 
- \frac{\sinh(\varepsilon (x_{i+1}-x))}{\sinh(\varepsilon h_i)}\frac{\sinh(\varepsilon 
x_i)\sinh(\varepsilon(1-x))}{\varepsilon\sinh(\varepsilon)}\nonumber\\
&=\frac{\sinh(\varepsilon(1-x))}{\varepsilon\sinh(\varepsilon)\sinh(\varepsilon h_i)}
\sinh(\varepsilon x_{i+1})\sinh(\varepsilon(x-x_i)),
\end{align}
where we used point \eqref{item:fourth_identity} of Lemma \ref{lemma:identities} in the last step.
Inserting \eqref{eq:power_tmp1} into \eqref{eq:power_tmp3}, we proceed to subtract the second term so that 
\begin{align*}
P_{1,\varepsilon,X}(x)^2 
&=\frac{\sinh(\varepsilon(1-x))\sinh(\varepsilon x_{i+1})\sinh(\varepsilon(x-x_i))}{\varepsilon\sinh(\varepsilon)\sinh(\varepsilon h_i)}
- \ell_{i+1}(x) k_{\beta, \varepsilon}(x, x_{i+1})\\
&=
\frac{\sinh(\varepsilon(x-x_i)}{\varepsilon\sinh(\varepsilon)\sinh(\varepsilon h_i)}
\left(\sinh(\varepsilon(1-x))\sinh(\varepsilon x_{i+1})- \sinh(\varepsilon x)\sinh(\varepsilon(1-x_{i+1}))\right)\\
&=
\frac{\sinh(\varepsilon(x-x_i)}{\varepsilon\sinh(\varepsilon)\sinh(\varepsilon h_i)}
\sinh(\varepsilon) \sinh(\varepsilon(x_{i+1}-x))
=
\frac{\sinh(\varepsilon(x-x_i)\sinh(\varepsilon(x_{i+1}-x))}{\varepsilon\sinh(\varepsilon h_i)},
\end{align*}
where we used again \eqref{item:fourth_identity} of Lemma \ref{lemma:identities}.
\end{proof}

We now use this formula to derive rates of convergence in terms of the fill distance
\begin{equation*}
h_X:=\max_{x\in(0,1)} \min_{0\leq i\leq N+1} \|x-x_i\|=\frac12\max\limits_{0\leq i\leq N} h_i,
\end{equation*}
which is commonly used in kernel interpolation, except that also the boundary points are included in the definition to account for the zero boundary conditions.

\begin{theorem}\label{thm:power}
For any $X\subset (0,1)$ and $\varepsilon>0$ we have 
\begin{equation}\label{eq:power_bound}
\norm{L_\infty(0,1)}{P_{1,\varepsilon,X}} 
= \max\limits_{0\leq i\leq N} \sqrt{\frac{\tanh(\varepsilon h_i/2)}{2 \varepsilon}}
= \sqrt{\frac{\tanh(\varepsilon h_X)}{2 \varepsilon}}
=\frac{1}{\sqrt{2}} h_X^{1/2}+ \mathcal O(\varepsilon h_X^{3/2}),
\end{equation}
where the $\mathcal O$ term contains exact orders in terms of both $\varepsilon$ and $h_X$.

In particular, for each $N\in\N$ the set $\overline X:=\left\{x_i:=i/(N+1)\right\}_{i=1}^N$ of equally spaces points is the unique mimimizer of
$\norm{L_\infty(0,1)}{P_{1,\varepsilon,X}}$, with a value
\begin{equation}\label{eq:power_equi_bound}
\norm{L_\infty(0,1)}{P_{1,\varepsilon,\overline X}} 
=\sqrt{\frac{\tanh(\varepsilon/(2(N+1))}{2 \varepsilon}}=\frac{1}{2}(N+1)^{-1/2} + \mathcal O(\varepsilon (N+1)^{-3/2}).
\end{equation}
\end{theorem}
\begin{proof}
If $x\in X$ then $P_{1,\varepsilon,X}(x)=0$ by definition. 
Given instead $x\in (0,1)\setminus X$, we have from \eqref{eq:power_explicit}, and using \eqref{item:nth_identity} of Lemma \ref{lemma:identities}, that 
\begin{equation*}
\partial_x P_{1,\varepsilon,X}(x)^2 
=\frac{\cosh(\varepsilon(x-x^-)\sinh(\varepsilon(x^+-x))-\sinh(\varepsilon(x-x^-)\cosh(\varepsilon(x^+-x))}{\varepsilon\sinh(\varepsilon (x^+ - x^-))}
= \frac{\sinh(\varepsilon(x^+ + x^- - 2x)}{\varepsilon\sinh(\varepsilon (x^+ - x^-))},
\end{equation*}
which is positive if $x< x^*:=(x^-+x^+)/2$, zero for $x= x^*$, and negative for $x>x^*$. It follows that $P_{1,\varepsilon,X}(x)$ has in 
$[x^-, x^+]$ a unique maximum in $x^*$, where it takes the value 
\begin{equation*}
P_{1,\varepsilon,X}(x^*)
=\left(\frac{\sinh(\varepsilon(x^+-x^-)/2)^2}{\varepsilon\sinh(\varepsilon (x^+ - x^-))}\right)^{1/2}
=\sqrt{\frac{\tanh(\varepsilon(x^+-x^-)/2)}{2 \varepsilon}},
\end{equation*}
where we used \eqref{item:fifth_identity} of Lemma \ref{lemma:identities}. 
The first two identities in \ref{eq:power_bound} then follow by taking the supremum over $x\in(0,1)$, and using the definition of $h_X$. The asymptotic rate 
follows instead from the Taylor approximation $\tanh(x)=x - x^3/3 + \mathcal O(x^5)$ and the fact that $\sqrt{a + b} \leq \sqrt{a} + \sqrt{b}$, $a,b>0$.
Finally, the term $\sqrt{\tanh(\varepsilon h_X/2)/(2 \varepsilon)}$ is increasing in $h_X$, and thus for a given $N$ it is minimized when $h_X$ takes its 
minimal value, which is attained by $N$ equally spaced points. Inserting the corresponding value $h_X = 1/(2(N+1))$ in \eqref{eq:power_bound} 
gives 
\eqref{eq:power_equi_bound}.
\end{proof}

Figure \eqref{fig:pf_viz} shows the value of $P_{1,\varepsilon,X}(x)$ for $x\in[1/3,2/3]$ and $\varepsilon=0,1,10,100$. In Figure~\ref{fig:pf_rates} we report 
instead the rate of decay of $\|P_{1,\varepsilon,\overline X}\|_{L_\infty}$ for the same $\varepsilon$ and for sets $\overline X$ of $N=1, 2,\dots, 10^3$ 
equally 
spaced points.
Observe that the power function is strictly decreasing with $\varepsilon$, and in particular for $\varepsilon'\geq \varepsilon > 0$ 
equation \eqref{eq:power_bound} 
gives that $\norm{L_\infty(0,1)}{P_{1,\varepsilon,X}} \leq c \sqrt{\varepsilon'/\varepsilon}\norm{L_\infty(0,1)}{P_{1,\varepsilon',X}}$. However, it can be 
proven (see Chapter 7 in \cite{Fasshauer2015}) that 
$\norm{1,\varepsilon}{u}^2
=\seminorm{H^1}{v}^2 + \varepsilon^2 \norm{L_2}{u}^2$,
and in particular for all $\varepsilon'\geq \varepsilon\geq 0$ there is a norm equivalence
$
\norm{\calh_{1, \varepsilon'}}{f}
\leq \norm{\calh_{1, \varepsilon'}}{f}
\leq \sqrt{\varepsilon'/\varepsilon} \norm{\calh_{1, \varepsilon}}{f}
$ for all $f\in H_0^1(0,1)$. In other words, the two terms in the right hand side of \eqref{eq:pf_bound} scale linearly with $\varepsilon$, bu in opposite 
directions.
\begin{figure}
 \centering
\begin{subfigure}{.45\textwidth}
\includegraphics[width=\textwidth]{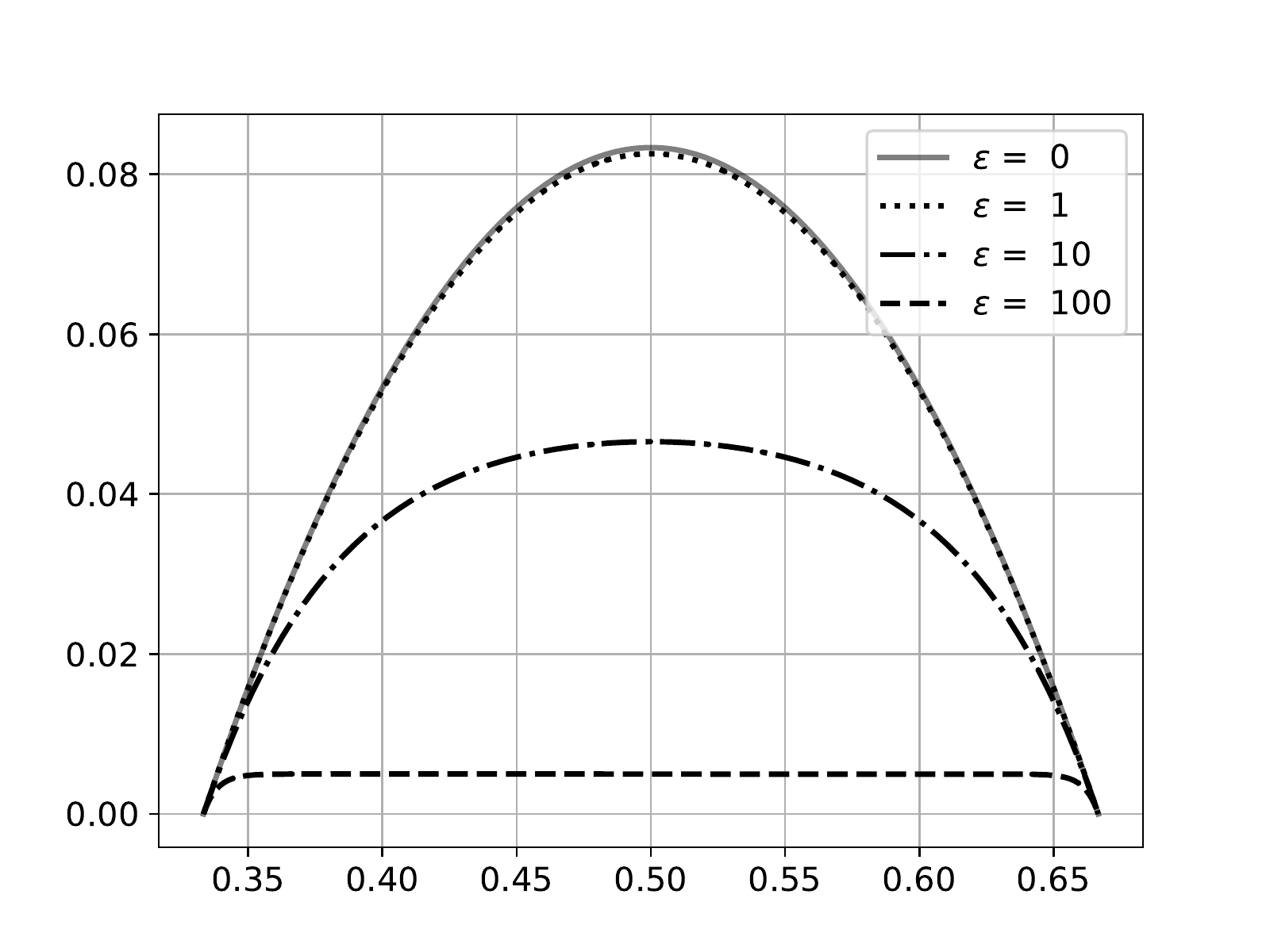}
\caption{}\label{fig:pf_viz}
\end{subfigure}
\begin{subfigure}{.45\textwidth}
\includegraphics[width=\textwidth]{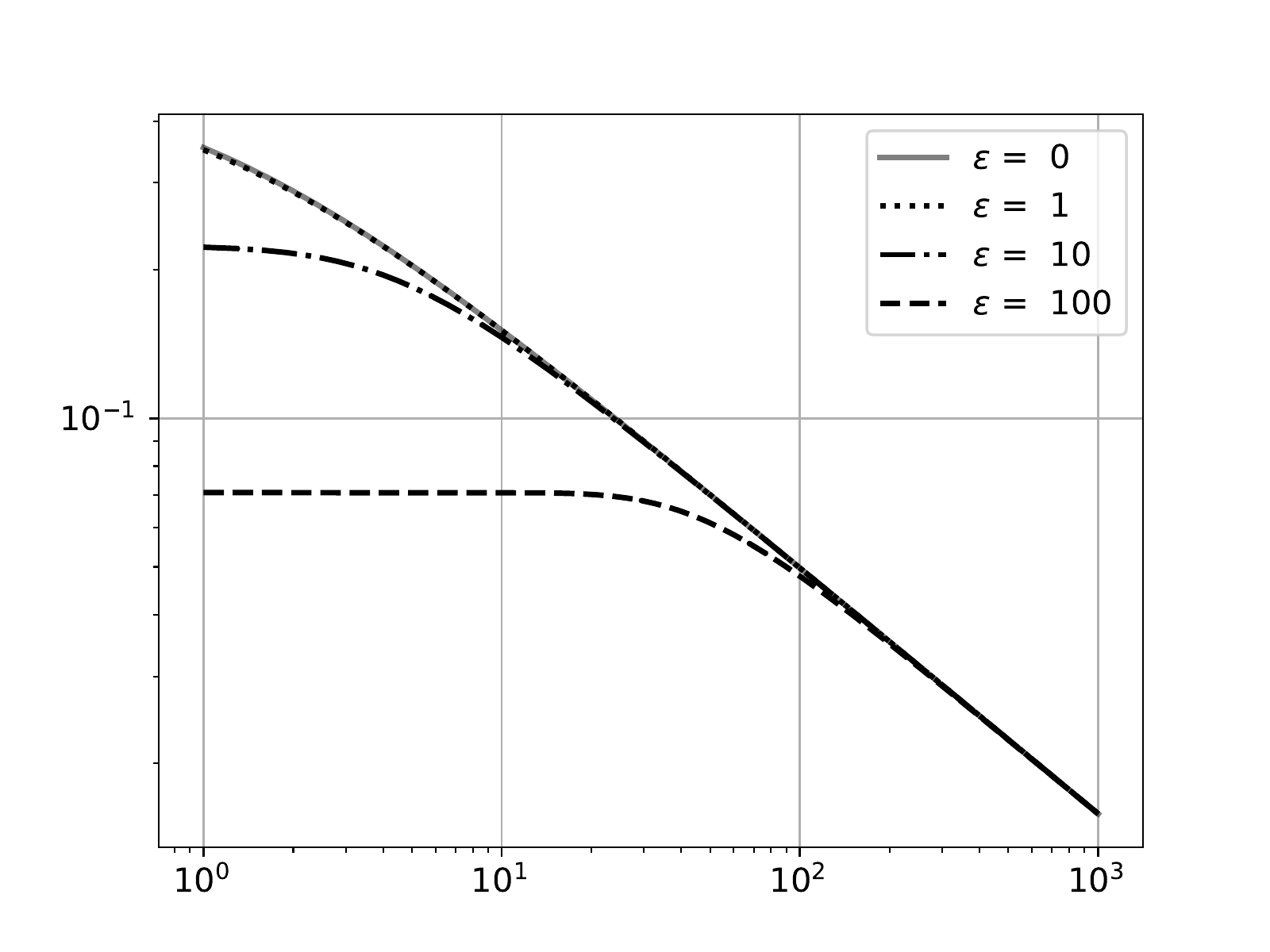}
\caption{}\label{fig:pf_rates}
\end{subfigure}

\caption{Power functions for $k_{1,\varepsilon}$ with $\varepsilon=0,1,10,100$: Visualization of the values of the power function in an example interval  
$x_i=1/3$, $x_{i+1}=2/3$ (Figure \ref{fig:pf_viz}), and decay of the maximum in $[0,1]$ of the power function corresponding to $N$ equally spaced points, for 
$N\leq 10^3$ (Figure~\ref{fig:pf_rates}).
}\label{fig:pf}
\end{figure}

\section{Flat limit}\label{sec:flat} 
As mentioned in Section ~\ref{sec:bb_basic}, for each $\beta\in\N$ interpolation with $k_{\beta, 0}$ coincides with piecewise spline interpolation with zero 
boundary conditions of order $\beta-1$, and the same approach used in Section~\ref{sec:b_splines} and Section \ref{sec:results} can be used in the case 
$\varepsilon=0$. 
In this case $\Null(L_{\be})=\Sp\{1, x, \dots, x^{2\beta-1}\}$, and the same construction as in Section \ref{sec:b_splines} leads to the usual tent functions
\begin{equation}\label{eq:leb_1_0}
\ell_{1,0, i}(x)=\begin{cases}
          (x - x_{i-1})/h_{i-1},& x_{i-1}\leq x<x_i\\
          (x_{i+1}-x)/h_i,& x_{i}\leq x<x_{i+1},
          \end{cases},\;\;1\leq i\leq N.
\end{equation}
Results on the Lebesgue and power functions can be obtained working as in Section \ref{sec:results}, or equivalently as limits for 
$\varepsilon\to0$ of the formula derived for $\varepsilon>0$. In particular Proposition \ref{prop:lebesgue} and \ref{thm:lebesgue} give 
the known fact that $\Lambda_{1, 0,X}(x)=1$ for $x\in 
[x_1,\dots, x_N]$. 
Taking the limit of \eqref{eq:power_bound} gives moreover
$\norm{L_\infty(0,1)}{P_{1,0,X}} = \max_{0\leq i\leq N} \sqrt{h_i/4} = \sqrt{h_X/2}$,
and again equally spaced points are optimal, with $\norm{L_\infty(0,1)}{P_{1,0,X}} =(N+1)^{-1/2}/2$.

Additionally, we can use the closed form expression of the cardinal basis to give a quantitative bound on the convergence of $I_{1,\varepsilon,X}$ to the 
limiting piecewise polynomial interpolant in the flat limit $\varepsilon\to0$ (see e.g. \cite{Lee2014,Song2012}). The result is interesting because it gives a 
bound that is explicit in 
terms of $h_X$, and not only of $\varepsilon$.

\begin{theorem}\label{thm:flat}
Let $X\subset(0,1)$ and $\varepsilon>0$. Then for all $f \in C(0,1)$ it holds
\begin{equation*}
\norm{L_\infty(0,1)}{I_{1,\varepsilon,X} f - I_{1,0,X} f} \leq \frac{8}{3} \varepsilon^2 h_X^2 \norm{\infty}{f_{|X}}.
\end{equation*}
\end{theorem}
\begin{proof}
For $x\in [x_{i-1},x_i]$ we write $\sinh(\varepsilon(x-x_{i-1})) = \varepsilon(x-x_{i-1}) + R(\xi_x)$ with $R(\xi_x):=\varepsilon^3(\xi_x-x_{i-1})^3/6\geq 0$ 
for some $\xi_x:=\xi(x)\in(x_{i-1},x)$. Using \eqref{eq:leb_1} and \eqref{eq:leb_1_0} we get
\begin{align*}
\ell_{1,\varepsilon,i}(x) -\ell_{1,0,i}(x)
&=\ell_{1,0,i}(x)\left(\frac{\varepsilon h_{i-1}}{\sinh(\varepsilon(x_i-x_{i-1})} - 1\right) + \frac{R(\xi_x)}{\sinh(\varepsilon(x_i-x_{i-1})}\\
&=\ell_{1,0,i}(x)\left(\frac{-R(\xi_{x_i})}{\varepsilon(x_i-x_{i-1}) + R(\xi_{x_i})}\right) + \frac{R(\xi_x)}{\varepsilon(x_i-x_{i-1}) + R(\xi_{x_i})}.
\end{align*}
Since $\xi_{x}, \xi_{x_i}\in (x_{i-1},x_i)$ and thus $R(\xi_{x_i}), R(\xi_x)\leq \varepsilon^3 h_i^3/6$, and $R(\xi_{x_i})\geq 0$, $\ell_{1, 0, 
i}\leq 1$, we have
\begin{equation*}
\left|\ell_{1,\varepsilon,i}(x) -\ell_{1,0,i}(x) \right|
\leq \left(\left|\ell_{1,0,i}(x)\right|+1\right)\left|\frac{\max(R(\xi_{x_i}), R(\xi_x))}{\varepsilon h_{i-1} + R(\xi_{x_i})}\right|
\leq 2\frac{\max(R(\xi_{x_i}), R(\xi_x))}{\varepsilon h_{i-1}}
\leq \frac{\varepsilon^2}{3} h_{i-1}^2
\leq \frac{4}{3}\varepsilon^2 h_{X}^2.
\end{equation*}
The result holds also for $[x_i, x_{i+1}]$ by symmetry. 
Using the cardinal form \eqref{eq:cardinal_interp} of $I_{1,\varepsilon,X}$, we compute
\begin{equation*}
\left|I_{1,\varepsilon,X} f(x) - I_{1,0,X} f(x) \right| 
\leq \norm{\infty}{f_{|X}} \sum_{i=1}^N\left|\ell_{1,\varepsilon,i}(x)-\ell_{1,0,i}(x)\right|
\leq \norm{\infty}{f_{|X}} \frac{8}{3} \varepsilon^2 h_X^2,
\end{equation*}
where we used the fact that at most two cardinal basis elements are nonzero at any given $x\in(0,1)$.
\end{proof}

\section{Conclusion}
In this paper we derived explicit formulas for the cardinal basis of a Brownian Bridge kernel $k_{1,\varepsilon}$, and we used them to prove stability and 
convergence results. In particular, we proved that equally spaced points are error-optimal for these kernels, and we provided a quantitative bound on the 
interpolation flat limit.
Future work will address the construction of $LB$-splines for $\beta>1$, and their use in the study of optimal point distributions. Comparison of this 
basis with the RBF-QR algorithm could also be of interest.

\bibliography{biblio}
\bibliographystyle{abbrv}

\end{document}